\documentclass[reqno]{amsart}

\usepackage{graphicx}
\usepackage{amssymb}
\usepackage[all]{xy}
\usepackage{enumitem}

 \newtheorem{theorem}{Theorem}[section]
 \newtheorem{proposition}[theorem]{Proposition}
 \newtheorem{corollary}[theorem]{Corollary}

 \newtheorem{lemma}[theorem]{Lemma}

 \newcommand{\rank}{\mathop{\rm rank}\nolimits}

 \newcommand{\ct}{\mathop{\rm ct}\nolimits}
 \newcommand{\cat}{\mathop{\rm cat}\nolimits}
 \newcommand{\lk}{\mathop{\rm lk}\nolimits}
 \newcommand{\hdim}{\mathop{\rm hdim}\nolimits}

 \newcommand{\ZZ}{{\mathbb{Z}}}

 \newcommand{\UU}{{\mathcal{U}}}

\begin{document}

\title[Triangulations with few vertices]{Triangulations with few vertices of manifolds with non-free fundamental group}  

\thanks{The author was supported by the Slovenian Research Agency research grant P1-0292 and research project J1-7025.}

\author{Petar Pave\v{s}i\'c}
\address[Petar Pave\v{s}i\'c]{\newline\hspace*{3mm} Faculty of Mathematics and Physics, University of Ljubljana, Slovenija}
\email{petar.pavesic@fmf.uni-lj.si}


\begin{abstract}
We study lower bounds for the number of vertices in a PL-trian\-gu\-la\-tion of a given manifold $M$.
While most of the previous estimates are based on the dimension and the connectivity of $M$,
we show that further information can be extracted by studying the structure of the fundamental group of $M$ 
and applying techniques from the Lusternik-Schnirelmann category theory. In particular, we prove that every PL-triangulation of 
a $d$-dimensional manifold ($d\ge 3$) whose fundamental group is not free has at least $3d+1$ vertices. 
As a corollary, every $d$-dimensional ($\ZZ_p$-)homology sphere that admits a PL-triangulation 
with less than $3d$ vertices is homeomorphic to $S^d$. Another important consequence is that every triangulation with small links of $M$ 
is combinatorial.
\ \\[3mm]
\end{abstract}

\maketitle


\section{Introduction and results} 

A \emph{triangulation} of a topological space $M$ is a simplicial complex $K$ together with a homeomorphism $M\approx |K|$ 
between $M$ and the geometric realization of $K$. If $M$ is a closed $d$-dimensional manifold then we are particularly 
interested in \emph{combinatorial} triangulations, where we require that the link (see bellow) of every simplex in $K$ is 
homeomorphic to a sphere.  A manifold admitting a combinatorial triangulation is called a \emph{PL-manifold}. 

Given a PL-manifold $M$, what is the minimal number of vertices in a combinatorial trian\-gu\-lation of $M$? This is a difficult question, 
because there are no standard constructions for triangulations with few vertices of a given manifold, nor there are sufficiently general 
methods to prove that some specific triangulation is in fact minimal. Apart from classical results on minimal triangulations 
of spheres and closed surfaces, and a special family of minimal triangulations 
for certain sphere bundles over a circle (so called Cs\'asz\'ar tori - see \cite{K}), there exists only a handful of examples for which 
the minimal triangulations are known. An exhaustive 
survey of the results and the existing literature on this problem can be found in \cite{L}. See also the recent article \cite{KN} which 
discusses a more general question of the number of faces in triangulations of manifolds and polytopes. 

Generally speaking, one may expect that the minimal number of vertices in a triangulation of a space increases with its complexity.
Most results that can be found in the literature use dimension, connectivity or Betti numbers of $M$ to express lower bounds 
for the number of vertices in a triangulation of $M$. In this paper we have been able to exploit the fundamental group 
and the Lusternik-Schnirelmann category to improve several estimates of the minimal number of vertices in a triangulations 
of a manifold.

In the rest of this section we state our main results.  In Section 2 we introduce and explain prerequisites on triangulations, 
Lusternik-Schnirelmann category and the covering type. Finally, in Section 3 we give the proofs of the theorems presented bellow.

Let us begin with a slight improvement of the theorem first proved by  Brehm and K\"uhnel \cite{BK}. Our approach is based on the notion 
of covering type \cite{GMP} and is much simpler than the original one. 
Recall that Poincar\'e duality together with the positive answer to the Poincar\'e conjecture imply 
that every simply-connected closed $d$-manifold, whose homology is trivial in dimensions less or equal to 
$d/2$ is homeomorphic to the $d$-sphere. For the remaining cases the minimal number of vertices in a 
triangulation can be estimated as follows. 

\begin{theorem}
\label{thm:simply connected}
Let $M$ be a simply-connected $d$-dimensional closed PL-manifold, and let $i$ be the minimal index for which $\widetilde H_i(M)\ne 0$. 
\begin{itemize}[leftmargin=7.5mm]   
\item[(a)] If $i=\frac{d}{2}$ then every combinatorial triangulation of $M$ has at least $\frac{3d}{2}+k+2$ vertices, where $k$ is 
the minimal integer for which ${i+k \choose i+1}\ge\rank H_i(M)$.
Moreover, $k$ can be equal to 1 only if $d\in\{2,4,8,16\}$.
\item[(b)] If $i<\frac{d}{2}$ then every combinatorial triangulation of $M$ has at least $2d-i+4$ vertices. 
\end{itemize}
In particular, every combinatorial triangulation of a closed, simply-connected $d$-manifold with at most $3\frac{d}{2}+2$ vertices 
represents the $d$-dimensional sphere.
\end{theorem}

The main contribution of this paper is the following theorem and its corollaries. In particular we obtain considerable improvements 
of estimates by Brehm-K\"uhnel \cite{BK} and Bagchi-Datta \cite{BD} of the number of vertices in PL-triangulations of homology spheres. 
By \cite[Corollary 2]{BK} every PL-triangulation of a non simply-connected $d$-manifold ($d\ge 3$) has at least $2d+3$ vertices. 
That the value cannot be improved in general is shown by K\"uhnel who constructed a family of 
$S^{d-1}$-bundles over the circle $S^1$ that admit PL-triangulations with $2d+3$ vertices. 
However, if the fundamental group of $M$ is not free, then we obtained a better estimate:

\begin{theorem}
\label{thm:non free pi1}
If $M$ is a $d$-dimensional ($d\ge 3$) closed manifold whose fundamental group is not free, then every combinatorial triangulation 
of $M$ has at least $3d+1$ vertices.
\end{theorem}

It is worth noting that closed 3-manifolds whose fundamental group is free are quite special, being either 
the 3-sphere or connected sums of tori $S^2\times S^1$ and twisted tori $S^2\underline{\times}\, S^1$. 
All the other closed 3-manifolds (in particular, all hyperbolic manifolds) satisfy the assumptions of the above theorem. 

An important family of examples whose fundamental group is not free are the \emph{homology spheres}, i.e., manifolds, whose homology 
groups vanish except in the top dimension, where the homology group is $\ZZ$. A simply-connected homology sphere is homeomorphic to 
a sphere by the positive answer to the Poincar\'e conjecture but for every $d\ge 3$ there exist $d$-dimensional homology spheres 
that are not homeomorphic to $S^d$. As the fundamental group of a homology sphere must be a perfect group, it cannot be free (unless it is trivial), 
therefore Theorem \ref{thm:non free pi1} implies the following improvement of the estimate in \cite[Corollary 4]{BK}. 

\begin{corollary}
\label{cor:homology sphere}
Every $d$-dimensional homology sphere that admits a combinatorial trian\-gu\-la\-tion with at most $3d$ vertices is a PL-sphere.
\end{corollary}

Bagchi and Datta \cite{BD} obtained an estimate of the minimal number of vertices in PL-triangulations of $\ZZ_2$-homology 
spheres, i.e., manifolds whose $\ZZ_2$-homology is isomorphic to that of a sphere (non-trivial examples are 3-dimensional 
odd lens spaces). Their results are improved (except in dimensions 3 and 4) by the following:

\begin{corollary}
\label{cor:Z2 homology sphere}
Every $d$-dimensional $\ZZ_p$-homology sphere that admits a combinatorial trian\-gu\-la\-tion with at most $3d$ vertices is a PL-sphere.
\end{corollary}

We conclude with a useful recognition criterion for combinatorial triangulations. 

\begin{theorem} \label{thm:small tgl}
Let $K$ be a triangulation of a $d$-dimensional manifold. If for every $k\ge 3$ the link of each simplex of codimension $k+1$ in $K$ 
has at most $3k$ vertices, then the triangulation $K$ is combinatorial.
\end{theorem}

%
%

\section{Preliminaries}

In this section we give recollect concepts and results that are needed in the proofs of above theorems. 

\subsection{Simplicial complexes and PL-triangulations} Here we describe two special constructions and refer the reader 
to the article of J. Bryant \cite{Bryant} for the definitions of triangulations, skeleta,  open and closed stars, links, joins, 
combinatorial triangulations and other standard concepts of PL-topology. 

Given a triangulation $M\approx |K|$ we identify the set of vertices of the triangulation with the $0$-skeleton $K^0$ of the simplicial
complex $K$. For a subset $V\subseteq K^0$, let $K(V)$ denote the full subcomplex of $K$ spanned by $V$, i.e. the maximal 
subcomplex of $K$ whose 0-skeleton is $V$. 
It is easy to check that for every vertex $v\in K^0$ the subcomplex $K(V\cup\{v\})$ can be obtained as the union of $K(V)$ and the 
join of $v$ with the part of the link of $v$ contained in $K(V)$, which can be expressed by the following formula:
\begin{equation}
K(V\cup\{v\})=K(V)\cup v*(\lk(v)\cap K(V))
\end{equation}

Furthermore, let us denote by $N(V)\subseteq |K|$ the union of open stars (with respect to $K$) of vertices in $V$. Clearly, the 
geometric realization $|K(V)|$ is a subspace of $N(V)$.

\begin{lemma}
\label{lem:KN}
$N(V)=|K|-|K(K^0-V)|$, therefore $N(V)$ is an open neighbourhood of $|K(V)|$ in $|K|$.
Moreover, $|K(V)|$ is a deformation retract of $N(V)$.
\end{lemma}
\begin{proof}
The first statement is obvious. In order to obtain a deformation retraction recall that every point $x\in|K|$ can be written 
uniquely in terms of barycentric coordinates 
$$x=\sum_{v\in K^0} \lambda_v(x)\cdot v.$$ 
By definition, for every $x\in N(V)$ there is at least one
$v\in V$ for which $\lambda_v(x)>0$, so we may define the retraction 
$$r\colon N(V)\to |K(V)| \ \ \text{as} \ \ r(x):=\sum_{v\in V} \lambda_v(x)\cdot v.$$
Clearly, $r$ is homotopic to the identity of $N(V)$ through a straight-line homotopy.
\end{proof}

In particular, if $K^0$ is partitioned into two disjoint subsets $V,V'$ then $N(V)$ and $N(V')$ form an open
cover of $|K|$ and 
$$N(V)\cap N(V')=N(V)-|K(V)|=N(V')-|K(V')|.$$ 

\begin{lemma}
\label{lem:complement homology}
Let $K$ be a combinatorial triangulation of a closed $d$-dimensional manifold. If $V\subseteq K^0$ spans a $d$-dimensional simplex in $K$ 
then $N(V)-|K(V)|$ is homotopy equivalent to a $(d-1)$-dimensional sphere and for $i<d$
$$H_i(K(K^0-V))\cong H_i(K)\ \ \text{and}\ \ H^i(K(K^0-V))\cong H^i(K)$$ 
(integer (co)homology unless $|K|$ is non-orientable and $i=d-1$, in which case one should use $\ZZ_2$-coefficients). 
\end{lemma}
\begin{proof}
The first claim follows easily by excision of the interior of the simplex $K(V)$.  
To prove the second statement for homology groups let $V'=K^0-V$ and consider the following portion of the Mayer-Vietoris sequence 
$$H_i(N(V)\cap N(V'))\to H_i(N(V))\oplus H_i(N(V'))\to H_i(K)\to H_{i-1}(N(V)\cap N(V'))$$
Observe that $H_i(N(V))=0$, that $H_i(N(V)\cap N(V'))=H_i(S^{d-1})=0$ for $i<d-1$, and that 
$H_d(|K|)\to H_{d-1}(N(V)\cap N(V'))$ is surjective (with $\ZZ_2$-coefficients if $|K|$ is non-orientable). \
By exactness of the above sequence 
$$H_i(K)\cong H_i(N(V'))\cong H_i(K(V'))$$ for $i<d$. The proof for cohomology groups is similar.
\end{proof}

\subsection{Lusternik-Schnirelmann category}

A subset $A\subseteq X$ of a topological space $X$ is said to be \emph{categorical} if the inclusion map $A\hookrightarrow X$ is nul-homotopic (i.e., if there exists a homotopy
between the inclusion and the constant map). The minimal cardinality of an open categorical cover of $X$ is denoted $\cat(X)$ and is called the \emph{Lusternik-Schnirelmann category} of $X$. 
For example, the category of a space is 1 if, and only if, it is contractible, and the category of a (non-contractible) suspension is 2, because every suspension
has a natural cover by two contractible cones. See \cite{CLOT} for a comprehensive survey of the results and the vast literature about Lusternik-Schnirelmann category and related topics. 
(Keep in mind when comparing the results that the survey \cite{CLOT} and the article \cite{DKR} use the normalized value of $\cat(X)$ which is one less than in our 
definition so that contractible spaces have category 0 and non-contractible suspensions have category 1). Lusternik-Schnirelmann category is tightly related to other homotopy invariants,
for example, a well-known result states that if $\cat(X)\le 2$ then the fundamental group of $X$ is free (see \cite[p.44]{CLOT}). 

We will base our results on a similar but much deeper theorem proved by Dranishnikov, Katz and Rudyak \cite[Corollary 1.2]{DKR}: if $M$ is a closed $d$-dimensional 
manifold ($d\ge 3$) and if $\cat(M)\le 3$ then the fundamental group of $M$ is free. 
Their proof is based on the notion of category weight which we briefly recall. Roughly speaking, a non-zero class $u\in \widetilde H^*(M)$ (here we omit the coefficients for 
cohomology from the notation) has \emph{category weight} at least $k$ if the restriction of $u$ to any union of $k$ categorical subsets of $M$ is trivial. 
Precise definition is slightly more technical - see \cite[Section 2.7]{CLOT} or \cite[Section 3]{DKR}. Clearly, if we can find classes $u,v\in \widetilde H^*(M)$ of 
weight $k$ and $l$ respectively, and such that $0\neq u\cdot v\in\widetilde H^*(M)$, then $\cat(M)>k+l$. We can summarize the main result of \cite[Section 4]{DKR} as follows:

\begin{theorem}
\label{thm:DKR}
Let $M$ be a closed $d$-dimensional ($d\ge 3$) manifold $M$ whose fundamental group is not free. Then there exist suitable systems of coefficients on $M$ and cohomology classes 
$u\in H^2(M)$ of weight 2 and $v\in H^{d-2}(M)$ of weight 1, such that $0\neq u\cdot v\in H^d(M)$. As a consequence, $\cat(M)\ge 4$.
\end{theorem}

\subsection{Homotopy triangulations and covering type}

Let us denote by $\Delta(X)$ the minimal number of vertices in a triangulation of a compact polyhedron. Clearly, $\Delta(X)$ is a topological invariant of compact
polyhedra but it is in general very far from being a homotopy invariant. As an easy example let $X_1:=S^1\vee S^1\vee S^1$, the one-point union of three circles,
let $X_2$ be the graph with two vertices and four parallel edges between them and let $X_3:=\Delta_3^{(1)}$, the 1-skeleton of the tetrahedron. All three spaces have the same
homotopy type and yet easy geometric reasoning shows that $\Delta(X_1)=7$, $\Delta(X_2)=5$, $\Delta(X_3)=4$. To obtain a homotopy invariant notion recall that 
a \emph{homotopy triangulation} of $X$ is a simplicial complex $K$ together with a homotopy equivalence $X\simeq |K|$. Then the minimal number of vertices among
all possible homotopy triangulations of $X$ is not only a homotopy invariant of $X$ but it also provides a link to the concept of covering type that was 
recently introduced by M.~Karoubi and C.~Weibel \cite{KW}. 

Recall that a cover $\UU$
of a space $X$ is said to be \emph{good} if all finite non-empty intersections of elements of $\UU$ are contractible. Standard examples are covers by convex sets,
covers of polyhedra by open stars of vertices and covers of Riemannian manifolds by geodesic balls. One of the main facts about good covers is the Nerve Theorem 
(see \cite[Corollary 4.G3]{Hatcher}): if $\UU$ is a good open cover of a paracompact space $X$, then $X\simeq |N(\UU)|$, where $|N(\UU)|$ is the geometric realization
of the nerve of $\UU$. Karoubi and Weibel defined the \emph{covering type} of $X$ as the minimum cardinality of a good open cover of a space that is homotopy equivalent to $X$. 

If $X$ admits a homotopy triangulation $X\simeq |K|$, where the simplicial complex $K$ has $n$ vertices, then the open stars of the vertices form a good cover for $|K|$,
therefore $\ct(X)\le n$. Conversely, if there exists a homotopy equivalence $X\simeq Y$ where $Y$ has a good open cover $\UU$ with 
$n$ elements, then $X\simeq Y\simeq |N(\UU)|$ is a homotopy triangulation of $X$ with $n$ vertices. 
Thus we have proved the following result (cf. \cite[Theorem 1.2]{GMP}):
\begin{proposition}
If $X$ has the homotopy type of a compact polyhedron, then $\ct(X)$ equals the minimal number of vertices in a homotopy triangulation of $X$.
\end{proposition}
For every compact polyhedron $X$ there is the obvious relation $\Delta(X)\ge\ct(X)$ and we have seen previously that $\Delta(X)$ can be in fact much bigger that $\ct(X)$. 
However, if $M$ is a closed triangulable manifold then there is some evidence that $\Delta(M)$ are $\ct(M)$ close and often equal. Notably, Borghini and Minian \cite{BM}
showed that for closed surfaces $\Delta(M)$ and $\ct(M)$ coincide, with the sole exception of the orientable surface of genus 2, where the two quantities differ by one. 

There are several useful estimates of $\ct(X)$ based on other homotopy invariants of $X$. For example, let $\hdim(X)$  denote the homotopy dimension of $X$, i.e. the minimal dimension
of a homotopy triangulation of $X$. Then we have the following estimate (cf. \cite[Proposition 3.1]{KW}):
\begin{proposition}
\label{prop:top homology}
Let $k=\hdim(X)$. If $\ct(X)=k+2$, then $X\simeq S^k$, otherwise $\ct(X)\ge k+3$.
\end{proposition} 
\begin{proof}
If $\ct(X)\le n$, then by Nerve theorem $X$ has a homotopy triangulation by a subcomplex of $\Delta_{n-1}$. However, $|\Delta_{n-1}|$ is contractible, the only 
subcomplex of $\Delta_{n-1}$ whose homotopy dimension is its $(n-2)$-skeleton $|\Delta_{n-1}^{(n-2)}|\approx S^{n-2}$, while all the other subcomplexes have the homotopy
dimension at most $n-3$. 
\end{proof}
Govc, Marzantowicz and Pave\v{s}i\'{c} \cite{GMP} applied techniques from Lusternik-Schnirelmann category to obtain further estimates of  the covering type of a space and proved 
the following results:
\begin{theorem}(\cite[Theorem 4.1]{GMP}) 
\label{thm:wedge of spheres}
The covering type of a $r$-fold wedge of sphere of dimension $i$ equals the minimal integer $n$ for which ${n-1\choose i+1}\ge r$.
\end{theorem}
and
\begin{theorem}(\cite[Corollary 2.4]{GMP})
Let $M$ be a $d$-dimensional closed manifold. Then every triangulation of $M$ has at least
$$ 1+d+ \frac{1}{2}\cat(M)(\cat(M)-1)$$ vertices.
\end{theorem}

\section{Proofs}

In this section we provide the proofs for the results states in Section 1.

\begin{proof} ({\bf of Theorem \ref{thm:simply connected}})
Observe that $M$ is by assumption simply-connected and hence orientable, which implies that Poincar\'e duality holds with arbitrary coefficients. 

Let $K$ be a combinatorial triangulation of $M$. Since $M$ is $d$-dimensional, there exists a $(d+1)$-element subset $V\subset K^0$ 
spanning a simplex.
Lemma \ref{lem:complement homology}, together with Seifert-van Kampen theorem imply that $K(K^0-V)$ is simply connected and that 
$H_i(K(K^0-V))=H_i(K)$ for $i<d$.

Under the assumption (a), if $d=2i$  and if $H_i(M)$ is the first non-trivial homology group of $M$, then the homology of $K(K^0-V)$  
is free and concentrated 
in dimension $i$. It follows that $K(K^0-V)$ is homotopy equivalent to a wedge of $i$-dimensional spheres. 
By Theorem \ref{thm:wedge of spheres} the covering type 
of a wedge of $r$ spheres of dimension $i$ is equal to $i+k+1$ where $k$ is the minimal integer satisfying ${i+k \choose i+1}\ge r$. 
We conclude that $K^0$ has at least $(d+1)+(\frac{d}{2}+k+1)=3\frac{d}{2}+k+2$ elements. 

Moreover, if $k=1$ then clearly $\rank H_i(M)=1$, therefore $M$ admits a CW-decomposition with three cells in dimensions $0,i$ 
and $d$, respectively. 
Then the $i$-dimensional skeleton is the sphere $S^i$ and the $d$-dimensional cell is attached to $S^i$ 
by a map with Hopf invariant 1. By the celebrated theorem of Adams, this is possible only if $i\in\{1,2,4,8\}$.

Under the assumption (b) $H_i(M)\ne 0$. If $H_i(M)\cong \ZZ$, then by the Universal Coefficients Theorem $H^i(M)\cong\ZZ$ and by Poincar\'e duality $H^{d-i}(M)\cong\ZZ$.
On the other hand if $H_i(M)\not\cong \ZZ$, then by Poincar\'e duality $H^{d-i}(M)\not\cong 0$ or $\ZZ$. 
Lemma \ref{lem:complement homology} yields $H^k(K(K^0-V))\cong H^k(M)$ for $k<d$, which in both cases implies that $\hdim(K(K^0-V))\ge d-i$, and that the cohomology of $K(K^0-V)$ is not that of a sphere. 
By  Proposition \ref{prop:top homology} the covering type of $K(K^0-V)$ is at least $d-i+3$. We conclude that $K^0$ has at 
least $(d+1)+(d-i+3)=2d-i+4$ elements.
\end{proof}

\begin{proof} ({\bf of Theorem \ref{thm:non free pi1}}) 
Let $K$ be a combinatorial triangulation of $M$. Since $M$ is $d$-dimensional, its triangulation must contain at least one $d$-simplex, 
and so there exist 
vertices $v_1,\ldots,v_{d+1}\in K^0$ that span a $d$-dimensional simplex in $K$.
Let us enumerate the remaining vertices so that $K^0=\{v_1,\ldots,v_{d+1},\ldots,v_n\}$. 

By adding one vertex at a time we obtain a sequence of subcomplexes 
$$\Delta_d=K_{d+1}<\ldots<K_k<K_n=K,$$
where $K_k=K(v_1,\ldots,v_k)\le K$. Since $\pi_1(M)$ is non-trivial, there exists a minimal $l$, such that $\pi_1(|K(v_1,\ldots,v_l)|)$ is non-trivial. 
By expressing $K_l$ as in formula (1) 
$$K_l=K_{l-1}\cup v_l*(\lk(v_l)\cap K_{l-1}),$$ 
we see that $K_l$ is a union of two simply-connected subcomplexes.
By Seifert-van Kampen theorem its fundamental group can be non-trivial only if (the geometric realization of) the intersection 
$L:=\lk(v_l)\cap K(v_1,\ldots,v_{l-1})$ has at least two components. Let us denote $L':=\lk(v_l)\cap K(v_{l+1},\ldots,v_n)$. Then $L$ and $L'$ are 
full subcomplexes of $\lk(v_l)$ and their vertices determine a partition of the vertices of $\lk(v_l)$. By lemma \ref{lem:KN} $|L'|$ is a deformation retract
of $|\lk(v_l)|-|L|$. Since $|\lk(v_l)|\approx S^{d-1}$, we can apply Alexander duality \cite[Theorem 3.44]{Hatcher} and obtain that
$H^{d-2}(|L'|)\cong \widetilde H_0(|L|)\ne 0.$ By Proposition \ref{prop:top homology} there exist $d-1$ vertices, which we may label as $v_{l+1},\ldots,v_{l+d-1}$, that span 
a simplex in $L'$. Since these vertices are contained in $\lk(v_l)$, they can be joined to $v_l$ in $K$, therefore vertices $v_{l},\ldots,v_{l+d-1}$ span a simplex in  $K$.  

Let us denote $A:=\{v_1,\ldots, v_{d+1}\}$ and $B:=\{v_l,\ldots,v_{l+d-1}\}$. $A$ and $B$ are disjoint and together contain $2d+1$ vertices of $K^0$. To conclude the proof,
we must show that $K^0-A-B$ contains at least $d$ vertices.
 
Since $\pi_1(M)$ is not free, Theorem \ref{thm:DKR} states that there exist cohomology classes $u\in H^2(M)$ of weight 2 and $v\in H^{d-2}(M)$ of weight 1, such that $u\cdot v\neq 0$.
Both $K(A)$ and $K(B)$ are contractible, therefore $N(A\cup B)=N(A)\cup N(B)$ is a union of two categorical sets. It follows that $u|_{N(A\cup B)}=0$, and so  
the restriction of $v$ to $N(K^0-A-B)$ cannot be trivial, as it would contradict $u\cdot v\neq 0$. Therefore $H^{d-2}(N(K^0-A-B))\neq 0$ and the Proposition \ref{prop:top homology} 
implies that $K^0-A-B$ must contain at least $d$ vertices, as claimed. 
\end{proof}

\begin{proof} ({\bf of Corollaries \ref{cor:homology sphere} and \ref{cor:Z2 homology sphere}})
Every 1- or 2-dimensional homology sphere is a PL-sphere so we may assume $d\ge 3$. If there is a PL-triangulation of $M$ with less than $3d+1$ vertices, 
then  $\pi_1(M)$ is free by by Theorem \ref{thm:non free pi1}. Therefore, assumptions $H_1(M)=0$ or $H_1(M;\ZZ_p)=0$ imply  that $M$ is simply-connected,
and so it is homeomorphic to $S^d$ by the positive answer to the Poincar\'e conjecture. 
\end{proof} 


\begin{proof} ({\bf of Theorem \ref{thm:small tgl}}) 
Let $\sigma$ be a simplex in $K$ of codimension $k+1$ and let $x\in |K|$ be a point lying in the interior of $\sigma$. Then we may use excision and homology sequence of the pair to
relate the homology of $\lk(\sigma)$ to the local homology of $|K|$ at $x$:
$$\ZZ\cong H_d(|K|,|K|-x)\cong H_d(\sigma*\lk(\sigma),\partial\sigma*\lk(\sigma))\cong $$
$$\cong H_{d-1}(\partial\sigma*\lk(\sigma))=H_{d-1}(\Sigma^{d-k-1}\lk(\sigma))\cong \widetilde H_k(\lk(\sigma)).$$
It follows that $\lk(\sigma)$ is a $k$-dimensional homology sphere. 

If codimension of $\sigma$ is at most $3$, then $\dim\lk(\sigma)\le 2$, therefore $\lk(\sigma)$ is a combinatorial triangulation of a sphere. We will use this as a base for 
induction. 

Let $\sigma$ be a simplex of codimension $k+1$, and assume that for each $v\in\lk(\sigma)$ the link $\lk(v,\lk(\sigma))=\lk(\{v\}\cup\sigma)$ is 
combinatorially equivalent to $S^{k-1}$. It follows that $\lk(\sigma)$ is a combinatorial triangulation of a $k$-dimensional homology sphere. By assumption 
$\lk(\sigma)$ has at most $3k$ vertices, so Corollary \ref{cor:homology sphere} implies that $\lk(\sigma)$ is a combinatorial triangulation of $S^k$. We conclude
that links of all simplices in $K$ are homeomorphic to spheres of suitable dimensions, hence the triangulation $K$ is combinatorial.
\end{proof}

\end{document}